\documentclass[12pt,a4paper,english]{article}

\newcommand{\beq}{\begin{equation}}
\newcommand{\eeq}{\end{equation}}

\usepackage[latin1]{inputenc}
\usepackage{babel}
\usepackage[centertags]{amsmath}
\usepackage{amsfonts}
\usepackage{amssymb}
\usepackage{color}
\usepackage{amsthm}
\usepackage{newlfont}
\usepackage{graphicx}
\usepackage{dsfont}
\usepackage{geometry}
\usepackage{mathrsfs}
\usepackage{hyperref}

\newtheorem{theorem}{Theorem}[section]

\newtheorem{coroll}[theorem]{Corollary}
\newtheorem{prop}[theorem]{Proposition}

\newtheorem{remark}[theorem]{Remark}
\newtheorem{ass}[theorem]{Assumption}

\newcommand{\msc}[1]{\textbf{MSC2010 Classification:} #1.}
\newcommand{\keywords}[1]{\textbf{Key words:} #1.}

\def\vs{\vspace}
\def\hs{\hspace}

\def\EE{\mathsf E}
\def\PP{\mathsf P}
\def\CC{\mathcal C}
\def\DD{\mathcal D}
\def\eps{\varepsilon}
\def\cF{{\cal F}}

\makeatletter  
\@addtoreset{equation}{section}

\makeatother  
\begin{document}
\title{\textbf{A note on the continuity of free-boundaries \\ in finite-horizon optimal stopping problems \\ for one dimensional diffusions}\footnote{This work was partially supported by EPSRC grant EP/K00557X/1.}}
\author{Tiziano De Angelis\thanks{School of Mathematics, University of Manchester, Oxford Rd.~M13 9PL Manchester, UK; \texttt{tiziano.deangelis@manchester.ac.uk}}}
\maketitle
\begin{abstract}
We provide sufficient conditions for the continuity of the free-boundary in a ge\-ne\-ral class of finite-horizon optimal stopping problems arising for instance in finance and economics. The underlying process is a strong solution of one dimensional, time-homogeneous stochastic differential equation (SDE). The proof relies on both analytic and probabilistic arguments and it is based on a contradiction scheme inspired by the maximum principle in partial differential equations (PDE) theory. Mild, local regularity of the coefficients of the SDE and smoothness of the gain function locally at the boundary are required. 
\end{abstract}
\msc{60G40, 60J60, 35R35, 35K20}
\vspace{+8pt}

\noindent\keywords{optimal stopping, one dimensional diffusions, free-boundary problems, continuous free-boundaries, second-order linear parabolic PDEs}
\section{Introduction}
In this work we provide some sufficient conditions for the continuity of optimal stopping boundaries in a class of optimal stopping problems of the form
\begin{align}\label{intr01}
\sup_{0\le \tau\le T-t}\EE_x\Big[G\big(t+\tau,X_\tau\big)\Big],
\end{align}
where $0<T<+\infty$ and the supremum is taken over stopping times of a Markov process $X$. The gain function $G$ is real valued and $X$ is the unique strong solution of a time-homogeneous, one dimensional stochastic differential equation (SDE). We require mild regularity on the coefficients of the SDE and some regularity properties of $G$ locally at the optimal boundary. In principle $G$ may even be discontinuous at points not on the optimal boundary without altering validity of our results.

It is well known that a \emph{link} exists between optimal stopping problems in probability theory and free-boundary problems in partial differential equations (PDE) theory.
For a general exposition of analytical and probabilistic results related to this topic one may refer, to \cite{BL82} and \cite{Fri73}, among others. Also, a probabilistic approach to optimal stopping and free-boundary problems may be found, for instance, in \cite{Shir} (mainly Markovian setting and infinite time-horizon) and \cite{Pes-Shir} (both Markovian setting and martingale methods with finite/infinite time-horizon) and references therein.

A particular class of problems that has been attracting intense studies for more than 40 years is the one in which $X$ is real valued and the free boundary is a curve $\{b(t),\,t\in[0,T]\}$ depending only on time. In these cases it is possible and often interesting to analyse the regularity of the map $t\mapsto b(t)$.

Properties of the free-boundary have been studied thoroughly by means of PDE me\-thods in a large number of cases related to the Stefan problem. There exists a huge literature in this area and listing a complete set of references is a demanding task that falls outside the scopes of this work (some insights may be found for instance in \cite{CDH67}, \cite{Sch76} and references therein). On the other hand, several works on stochastic control problems are partially devoted to the analysis of continuity and differentiability of the free-boundary by means of variational methods (cf.~\cite{BF77} and \cite{Fri75}, among others); in \cite{Fri75} for instance the author proves free-boundary's differentiability for a particular class of problems and shows its continuity in wide generality (cf.~Section 6 of that work).

Methods involving both analytical and probabilistic tools where originally developed by \cite{Kot73}, \cite{McK65} and \cite{VM}, among others, where differentiability of the free-boundary $b$ in the open interval $(0,T)$ was proven under suitable assumptions on $X$ and $G$. In \cite{McK65} it was also shown that $b$ solves a countable system of non-linear integral equation but the problem of uniqueness of such solution was left open. Another integral equation for $b$ and its first derivative $b^\prime$ was obtained in \cite{VM}; however, a full proof of existence and local uniqueness of a solution pair $(b\,,b^\prime)$ was only given in the case of $b^\prime$ bounded on $[0,T]$. This condition is somewhat restrictive in general and does not hold for instance in the famous example of the American put option. The latter has received significant attention due to its strongly applicative nature and the regularity of the associated free-boundary was analysed carefully (cf.~\cite{Bl06}, \cite{CC07}, \cite{CCJZ08} and \cite{Eks04}, among others, for an overview of known results in that setting). At the beginning of the 1990s an integral equation for the optimal boundary (with no derivatives) was derived independently by \cite{CJM92}, \cite{Jacka} and \cite{KIM} but, as it was pointed out in \cite{Myn92}, the question of uniqueness was left open at that time. More than ten years later it was proven in \cite{Pe05} that the free-boundary is the unique solution in the class of \emph{continuous} functions to this integral equation.

It seems then natural from the standpoint of optimal stopping to investigate continuity properties of the free-boundary. In fact, in a very large class of examples, if continuity is established \emph{a priori}, one may rely upon on an extension of It\^o's formula (cf.\ \cite{Pe05b} or \cite{Pe07} for a detailed exposition) to find an integral equation for $b$; uniqueness of its solution in the class of continuous functions may be then proven by developing arguments as in \cite{Pe05} (see \cite{Pes-Shir} and the references therein). This characterises the free-boundary unambiguously and the value $V$ of \eqref{intr01} may be expressed as a functional of $b$ itself.

A proof of continuity generally requires techniques based on \emph{ad hoc} arguments that have to be found on a case by case basis.  In fact, in optimal stopping literature this is usually obtained by an application of Newton-Leibnitz formula, combined with the so-called \emph{smooth-fit} property (i.e.~the fact that $V(t,\,\cdot\,)$ is $C^1$ across the optimal boundary) and estimates on $V$ obtained \emph{ex ante} (see \cite{Pes-Shir} for a list of examples).


Although Newton-Leibnitz formula turns out to be a suitable tool to deal with most of the examples that we could find in literature, we observed that some cases seem quite hard to tackle this way (cf.~for instance \cite{Ch-Haus09}, \cite{DeA-Fe13}, \cite{YYZ12} or \cite{CDeAV13}; in particular in \cite{DeA-Fe13} one may find applications of results of this work to zero-sum optimal stopping games). In fact, some difficulties may arise when one or more of the following facts occur: $i)$ $V$ is not convex/concave with respect to the space variable, $ii)$ the explicit expression of the process $X$ is unknown or the coefficients of its infinitesimal generator are non-trivial and make some estimates rather difficult, $iii)$ the gain function underlying the optimal stopping problem is non-differentiable or it is explicitly time-dependent, $iv)$ the free-boundary is non-monotone.

The purpose of this work is to provide an alternative proof of the free-boundary's continuity partly based on local regularity of $G$ at the boundary and partly based on properties of $V$ which are generally obtainable via probabilistic arguments. To the best of our knowledge all examples where continuity has already been established meet requirements of our setting. It might be worth noticing that smooth-fit condition is not needed in our proofs.

The rest of the paper is organised as follows. In Section 2 we introduce the optimal stopping problem, some standard assumptions and a list of conditions that will be used only when needed in different proofs of free-boundary's continuity. We take $X$ as the unique strong solution of a time-homogeneous SDE in $\mathbb{R}$ with locally Lipschitz coefficients, we assume that a free-boundary exists and we make some mild regularity assumptions on the gain function $G$ locally at the boundary. These assumptions are mainly in the spirit of a probabilistic approach to optimal stopping, rather than a PDE one. In Section 3 we prove that the free-boundary of an optimal stopping problem of this kind is continuous in all intervals where it is either increasing or decreasing (increasing here means $b(t_1)\le b(t_2)$, for $t_1\le t_2$). Proofs are provided for two different settings: firstly, we assume that $V$ has a local modulus of continuity; secondly, we replace that assumption by a suitable integrability condition on $G$ and extend continuity of $b$ to a setting in which $V$ is only continuous. A contradiction scheme and arguments inspired by the maximum principle are combined with probabilistic estimates to obtain the results. The work is completed by an alternative proof of the free-boundary's continuity in the special case of time independent gain functions.

\section{Setting and Assumptions}
Consider a complete probability space $(\Omega,\cF,\PP)$ equipped with the natural filtration $\mathbb{F}:=(\cF_t)_{t\ge 0}$ generated by a one dimensional, standard Brownian motion $B:=(B_t)_{t\ge0}$. Assume that the filtration is completed with $\PP$-null sets and it is therefore continuous. Without loss of generality we may consider $\Omega=C([0,\infty))$, i.e.~the canonical space of continuous trajectories and $\PP$ the Wiener measure; then $B_t=\omega(t)$ coincides with the coordinate mapping.

Results of this work hold for diffusions with state space in an arbitrary open subset $\mathcal{O}$ of $\mathbb{R}$ and such that conditions below hold in $\mathcal{O}$ (a simple example is the Geometric Brownian motion and  $\mathcal{O}:=(0,+\infty)$). It should be noticed that all proofs of Section 3 may be easily repeated in case of a diffusion taking values in $\overline{\mathcal{O}}$ provided that the free-boundary does not intersect $\partial\mathcal{O}$. Alternatively, if a portion of the free-boundary coincides with a portion of $\partial\mathcal{O}$, the analysis becomes more delicate and one may have to take into account for the behaviour of the diffusion at $\partial\mathcal{O}$. However, bearing this in mind we take $\mathcal{O}=\mathbb{R}$ to simplify the exposition.

Let functions $\mu:\mathbb{R}\to\mathbb{R}$ and $\sigma:\mathbb{R}\to\mathbb{R}_+$ be such that
\begin{itemize}
\item[(\textbf{A.1})]\hs{+20pt}\emph{$\mu$ and $\sigma$ are locally Lipschitz, $\mu$ is piecewise-$C^2$, $\sigma>0$ and it is piecewise-$C^3$.}
\end{itemize}
Denote by $X:=(X_t)_{t\ge0}$ the time-homogeneous real process that uniquely solves
\begin{align}\label{SDE}
dX_t=\mu(X_t)dt+\sigma(X_t)dB_t,\qquad X_0=x
\end{align}
in the strong sense. We denote by $\PP_x$ the probability measure induced by $X$ started at time zero from $x$.

Fix $T>0$ and take $G:[0,T]\times\mathbb{R}\to\mathbb{R}$ such that
\begin{itemize}
\item[(\textbf{A.2})]\hs{+115pt}$G$ \emph{is upper semi-continuous.} 
\end{itemize}
We can now define a general optimal stopping problem with value function given by
\begin{align}\label{Vfun}
V(t,x):=\sup_{0\le \tau\le T-t}\EE_x\Big[G(t+\tau,X_\tau)\Big],
\end{align}
where the supremum is taken over all $\mathbb{F}$-stopping times in $[0,T-t]$.

In many cases of interest one may verify that
\begin{itemize}
\item[(\textbf{A.3})]\hs{+110pt}\emph{$V$ is continuous on $(0,T)\times\mathbb{R}$}
\end{itemize}
and the stopping time
\begin{align}\label{taustar}
\tau_*(t,x):=\inf\big\{s\ge0\,:\,V(t+s,X_s)=G(t+s,X_s)\big\}\wedge(T-t)\quad\textrm{under $\PP_x$}
\end{align}
is optimal for \eqref{Vfun}.
The state space is then naturally split into a \emph{continuation set} $\CC:=\big\{V>G\big\}$ and a \emph{stopping set} $\DD:=\big\{V=G\big\}$ which are an open and a closed subset of $[0,T]\times\mathbb{R}$, respectively.

Define the infinitesimal generator $\mathbb{L}_X$ of $X$ by
\begin{align}\label{infgen}
\mathbb{L}_X f(x):=\frac{\sigma^2(x)}{2}f^{\prime\prime}(x)+\mu(x)f^\prime(x)\qquad\textrm{for $f\in C^2_b(\mathbb{R})$}.
\end{align}
From standard Markovian arguments and with no further assumptions one obtains that $V\in C^{1,2}$ inside $\CC$ and it solves the free-boundary problem
\begin{align}\label{freeb-pr}
\begin{array}{cl}
\vs{+5pt}
V_t+\mathbb{L}_XV=0 & \textrm{in $\CC$}\\
\vs{+5pt}
V_t+\mathbb{L}_XV\le 0 & \textrm{in $[0,T]\times\mathbb{R}$}\\
\vs{+5pt}
V\ge G & \textrm{in $[0,T]\times\mathbb{R}$}\\
\vs{+5pt}
V=G & \textrm{in $\DD\cup\{T\}\times\mathbb{R}$.}
\end{array}
\end{align}

We make now a basic existence assumption for the free-boundary.
\begin{ass}\label{assfreeb}
There exists a free-boundary $\{b(t),\,0\le t\le T\}$ such that
\begin{align}\label{cont-stop}
\CC:=\big\{x\in\mathbb{R}\,:\,x>b(t),\:t\in[0,T)\big\}\:\:\:\:\&\:\:\:\:\DD:=\big\{x\in\mathbb{R}\,:\,x\le b(t),\:t\in[0,T)\big\}\cup\{T\}\times\mathbb{R}.
\end{align}
\end{ass}
\noindent Although in general $G$ is only USC on $[0,T]\times\mathbb{R}$, one may verify that in numerous examples a stronger local regularity holds at points of the boundary (cf.~\cite{Pes-Shir} for an overview). Therefore we may consider $G$ fulfilling (\textbf{A.2}) and such that
\begin{itemize}
\item[(\textbf{A.4})]\emph{For any $t\in[0,T)$, there exists $\eps>0$, a ball $\mathcal{B}:=\mathcal{B}^{\eps}_{t,b(t)}$ centered in $(t,b(t))$ and with radius $\eps$ such that $G\in C^{1,2}(\overline{\mathcal{B}\cap\CC})$.}
\end{itemize}

Now we list conditions that will be used in Section 3 below to prove continuity of the free-boundary in different settings. It is important to stress that we will not employ all of them at the same time and that they will be recalled only when needed. The next two conditions are useful to show continuity of the free-boundary when it is increasing or decreasing, respectively.
\begin{itemize}
\item[(\textbf{C.1})] \emph{There exist $\eps>0$ and $\mathcal{B}$ as in (\textbf{A.4}) such that
\begin{align}\label{regG-bdry}
H:= G_t+\mathbb{L}_XG\le-\ell_\eps\quad\textrm{in $\overline{\mathcal{B}\cap\CC}$}
\end{align}
with $\ell_\eps>0$ a suitable constant.}
\end{itemize}

\begin{itemize}
\item[(\textbf{C.2})]\emph{
There exist $\eps>0$ and $\mathcal{B}$ as in (\textbf{A.4}) such that $G_{tx}$ and $G_{xxx}$ exist and are continuous in $\mathcal{B}\cap\CC$. Moreover, there exists a constant $\ell^\prime_\eps>0$ such that
\begin{align}\label{regH-bdry}
H_x:= \frac{\partial}{\partial x}\big(G_{t}+\mathbb{L}_XG\big)\ge\ell^\prime_\eps\quad\textrm{in $\overline{\mathcal{B}\cap\CC}$}
\end{align}
and
\begin{align}\label{VxGx}
V_x\ge G_x\quad\textrm{in $\mathcal{B}\cap\CC$}.
\end{align}}
\end{itemize}

\begin{remark}
Assumption \ref{assfreeb} is not binding. In fact, results of this work extend to the case of a free-boundary $\{c(t),\,0\le t\le T\}$ such that
\begin{align}\label{cont-stop2}
\hs{-5pt}\CC:=\big\{x\in\mathbb{R}\,:\,x<c(t),\:t\in[0,T)\big\}\:\:\:\:\&\:\:\:\:\DD:=\big\{x\in\mathbb{R}\,:\,x\ge c(t),\:t\in[0,T)\big\}\cup\{T\}\times\mathbb{R}.
\end{align}
Conditions $(\textbf{C.1})$ and $(\textbf{C.2})$ (with $\ell^\prime_\eps<0$ and reverse inequalities in \eqref{regH-bdry} and \eqref{VxGx} if \eqref{cont-stop2} holds) are instead crucial, as it will be shown in a counterexample below. Suitable extension to the case of multiple free-boundaries as for instance in \cite{DeA-Fe13}, \cite{DuT-Pe07} and \cite{YYZ12} may be obtained with minor modifications.
\end{remark}

Condition (\textbf{A.4}), (\textbf{C.1}), (\textbf{C.2}) and Assumption \ref{assfreeb} are mostly in the spirit of a probabilistic approach to optimal stopping problems rather than a PDE one. In fact, existence of a free-boundary as in \eqref{cont-stop}, as well as the structure of $\CC$ and $\DD$, may be proven quite often by means of techniques borrowed from stochastic calculus (again we refer the reader to \cite{Pes-Shir} for a list of examples). On the other hand, condition (\textbf{A.4}) and \eqref{regG-bdry} mean, roughly speaking, that the free-boundary lies in a portion of the $(t,x)$-plane where $G$ is smooth and it is locally not convenient to wait. Note that \eqref{regG-bdry} is stronger than the usual condition
\begin{align}
\DD&\subset\big\{(t,x)\,:\,G_t+\mathbb{L}_X G\le 0\big\}\qquad\textrm{for $G\in C^{1,2}$}\label{LGneg}
\end{align}
since it requires strict inequality. Local smoothness of $G$ at the boundary has been observed, for example, in problems related to mathematical finance (cf.~\cite{Jacka} or \cite{PeUys05}, among others). 
Conditions (\textbf{A.4}) and (\textbf{C.2}) are closely related as well. In fact regularity of $V$ in $\CC$, (\textbf{A.4}) and Assumption \ref{assfreeb} imply \eqref{VxGx}. On the other hand, if $G\in C^{1,2}$ and \eqref{regH-bdry} holds in $[0,T]\times\mathbb{R}$ one may show that Assumption \ref{assfreeb} and \eqref{VxGx} hold as well (cf.~for instance \cite{JL92}, Theorem 4.3).

As already mentioned conditions \eqref{regG-bdry} and \eqref{regH-bdry} are crucial in our proofs. In fact, continuity of the boundary may break down if we omit one of them. There is for instance an interesting counterexample in which at some points on the boundary $\ell_\eps=\ell^\prime_\eps\equiv0$ for all $\eps>0$. In Remark 15, at the end of \cite{Cox-Pe13}, authors analyse a particular optimal stopping problem for Brownian motion which is relevant to the study of Skorokhod embedding (cf.~eq.~(5.23) therein). The continuation set and the stopping set are se\-pa\-ra\-ted by discontinuous optimal boundaries (this was also shown in \cite{McCon91} with a different approach). The gain function has the particular form $G(x):=|x|-2\int^x_0{F(y)dy}$, where $F$, in principle, is not even $C^1$ (for simplicity we will consider $F\in C^1$). It was shown in \cite{Cox-Pe13}, Proposition 7, that discontinuities for a right-continuous, increasing boundary only happen at particular points $t_0$ such that $F^\prime(x)\equiv0$ for $x\in(b(t_0),b(t_0+))$. Equivalently, discontinuities for a right-continuous, decreasing boundary only happen for $t^\prime_0$ such that $F^\prime(x)\equiv0$ for $x\in(c(t^\prime_0+),c(t^\prime_0))$. However, for $t_0$ and $t^\prime_0$ as above one has $\mathbb{L}_XG(x)=\frac{1}{2}G_{xx}(x)=-F^\prime(x)\equiv0$ on $(b(t_0),b(t_0+))\setminus\{0\}$ and $(c(t^\prime_0+),c(t^\prime_0))\setminus\{0\}$ implying that both $\ell_\eps$ and $\ell^\prime_\eps$ in \eqref{regG-bdry} and \eqref{regH-bdry} cannot be larger than zero.


\begin{remark}
Once the existence of a free-boundary is proven one would like to add to \eqref{freeb-pr} the so-called \emph{smooth-fit} condition, i.e.~$V_x(t,b(t)+)=G_x(t,b(t)-)$, $t\in[0,T)$. This can hardly be done by probabilistic methods when an explicit solution of \eqref{SDE} is not known. However, under some additional assumptions on $\mu$ and $\sigma$ one could rely on results about variational inequalities and Sobolev embedding theorems (cf.~for instance \cite{Fri73} and \cite{Fri75}) to retrieve this further condition. For the purpose of this exposition the smooth-fit is not necessary, hence we will not discuss it here.
\end{remark}

We now introduce two conditions, each of which is sufficient, together with (\textbf{C.1}), to prove continuity of increasing free-boundaries. The first one is on the regularity of the value function \eqref{Vfun} and it is stronger than $(\textbf{A.3})$ but it may be verified in several examples of interest (cf.\ for instance \cite{DuT-Pe07} and \cite{Jacka}).
\begin{itemize}
\item[(\textbf{C.3})] \emph{The value function fulfils $(\textbf{A.3})$ with a local modulus of continuity. In particular there exists $\alpha>0$ and continuous functions $\theta_i:\mathbb{R}_+\to\mathbb{R}_+$ for $i=1,2$ such that
\begin{align}\label{Vcont}
\big|V(t+h,x+h^\prime)-V(t,x)\big|\le \theta_1(|x|)|h|^{\alpha/2}+\theta_2(|t|)|h^\prime|^\alpha
\end{align}
for $(t,x)\in[0,T-h]\times\mathbb{R}$ and $h,h^\prime>0$.}
\end{itemize}
\noindent Sufficient conditions for (\textbf{C.3}) to hold with $\theta_1=\theta_2\equiv const.$~are that $\mu,\sigma$ in \eqref{SDE} are Lipschitz and $|G(t_1,x_1)-G(t_2,x_2)|\le C(|t_2-t_1|^{\alpha/2}+|x_2-x_1|^\alpha)$ for a fixed constant $C>0$.

The second condition is very similar to a usual sufficient condition for the well-posedness of the optimal stopping problem (cf.\ for instance \cite{Pes-Shir}, Section 2.2).
\begin{itemize}
\item[(\textbf{C.4})] \emph{There exists $\delta>1$ such that the map $(t,x)\mapsto \kappa(t,x)$ defined by
\begin{align}\label{integr-G}
\kappa(t,x):=\EE_x\Big[\sup_{0\le s\le T-t}\big|G(t+s,X_s)\big|^\delta\Big]
\end{align}
satisfies
\begin{align}
\sup_{0\le t\le T}\int_{-R}^R\big|\kappa(t,x)\big|^{\frac{1}{\delta}}dx<\infty\qquad\textrm{for any fixed $R>0$}.
\end{align}}
\end{itemize}
The usual assumption only requires that the expression in \eqref{integr-G} is finite with $\delta=1$; however, one may often show that the map $(t,x)\mapsto \kappa(t,x)$ is in fact bounded on any compact set $[0,T]\times[-R,R]$.

It is worth noticing that conditions $(\textbf{A.1})$, $(\textbf{A.2})$ and $(\textbf{A.4})$ have a local character and $G$ could exhibit jumps somewhere without invalidating results of this paper. 


\section{Continuity of the free-boundary}

In what follows we will prove continuity of the free-boundary in different settings. For increasing free-boundary we obtain a first proof under condition (\textbf{C.3}) and a second one under condition (\textbf{C.4}); for decreasing free-boundary we only require condition (\textbf{C.2}).

\begin{theorem}\label{th-cont}
Let (\textbf{A.1})--(\textbf{A.4}) and Assumption \ref{assfreeb} hold. Let $[t_1,t_2]\subset [0,T]$ be a time interval where the free-boundary is \emph{increasing}. Then, under conditions (\textbf{C.1}) and (\textbf{C.3}), the free-boundary $t\mapsto b(t)$ is continuous on $[t_1,t_2]$.
\end{theorem}
\begin{proof}

Since $\DD$ is closed and $b$ is increasing, it follows from standard arguments (see, e.g.~\cite{Jacka}) that $b$ is \emph{right-continuous}. To prove continuity we argue by contradiction and assume that there exists $t_0\in(t_1,t_2]$ such that a discontinuity of $b$ occurs. That is, at $t_0$ one has $b(t_0-)<b(t_0)$, where $b(t_0-)$ denotes the left limit of the boundary at $t_0$ (this always exists as $b(t)$ is monotone increasing in $[t_1,t_2]$). Take $x_1$ and $x_2$ such that $b(t_0-)<x_1<x_2<b(t_0)$; then, for arbitrary but fixed $t^\prime\in(t_1,t_0)$ define an open bounded domain $\mathcal{R}\subset \CC$ with $\mathcal{R}:=(t^\prime,t_0)\times(x_1,x_2)$. Its parabolic boundary $\partial_P\mathcal{R}$ is formed by the horizontal lines $[t^\prime,t_0)\times\{x_i\}$, $i=1,2$, and by the vertical line $\{t_0\}\times[x_1,x_2]$ (note that in this setting $\CC$ lies on the left of the vertical segment $[b(t_0-),b(t_0)]$).

From \eqref{freeb-pr} we know that $V$ (uniquely) solves the Cauchy-Dirichlet problem
\begin{eqnarray}\label{continuity01}
\begin{array}{ll}
\vspace{+5pt}
u_t+\mathbb{L}_Xu=0 & \textrm{in $\mathcal{R}$}\\
\vspace{+5pt}
u=V & \textrm{on $\partial_P\mathcal{R}$}
\end{array}
\end{eqnarray}
and it is $C^{1,2}$ in the interior of $\mathcal{R}$. Denote by $C^{\infty}_c([x_1,x_2])$ the set of functions with infinitely many continuous derivatives and compact support in $[x_1,x_2]$. Take $\psi\ge0$ arbitrary in $C^{\infty}_c([x_1,x_2])$ and such that $\int^{x_2}_{x_1}{\psi(y)dy}=1$. Multiply the first equation in \eqref{continuity01} (with $V$ instead of $u$) by $\psi$ and integrate\footnote{Note that $V$ is $C^{1,2}$ on $\mathcal{R}_P$ and not necessarily on $\overline{\mathcal{R}}_P$. The integration with respect to $y$ may be interpreted in the sense of distributions by taking derivatives of $\psi$. The integral with respect to $s$ is well defined since $\int^{t_0}_{t}\int_{x_1}^{x_2}{V_t(s,y)\psi(y)dy\,ds}=\int_{x_1}^{x_2}{\big(V(t_0,y)-V(t,y)\big)\psi(y)dy}$.} over $(t,t_0)\times(x_1,x_2)$ for some $t\in(t^\prime,t_0)$. It gives
\begin{align}\label{continuity02}
\int^{t_0}_{t}\hs{-5pt}\int_{x_1}^{x_2}{\hs{-3pt}\Big(V_t(s,y)+\mathbb{L}_XV(s,y)\Big)\psi(y)dy\,ds}=0.
\end{align}
We want to estimate the left-hand side of \eqref{continuity02} and provide an upper bound. In order to do so we will study separately the two terms of the integrand.

Integrating $V_t$ over $(t,t_0)$ and using $V(t_0,y)=G(t_0,y)$ and $V(t,y)\ge G(t,y)$ it follows
\begin{align}\label{continuity03}
\int^{t_0}_{t}\hs{-5pt}\int_{x_1}^{x_2}{V_t(s,y)\psi(y)dy\,ds}\le \int_{x_1}^{x_2}{\big(G(t_0,y)-G(t,y)\big)\psi(y)dy}=\int^{t_0}_{t}\hs{-5pt}\int_{x_1}^{x_2}{G_t(s,y)\psi(y)dy\,ds}.
\end{align}
Now, we integrate by parts in $dy$ the term in \eqref{continuity02} involving the infinitesimal generator of $X$ and we use the fact that $\psi$ has compact support in $[x_1,x_2]$. It gives
\begin{align}\label{continuity04}
\int^{t_0}_{t}\hs{-5pt}\int_{x_1}^{x_2}{\mathbb{L}_XV(s,y)\psi(y)dy\,ds}=\int^{t_0}_{t}\hs{-5pt}\int_{x_1}^{x_2}{V(s,y)
{\mathbb{L}_X}^*\psi(y)dy\,ds},
\end{align}
where ${\mathbb{L}_X}^*$ denotes the formal adjoint of $\mathbb{L}_X$ defined by
\begin{align}\label{continuity05}
{\mathbb{L}_X}^*\phi(x):=\frac{1}{2}\frac{\partial^2}{\partial x^2}\big(\sigma^2(x)\phi(x)\big)-\frac{\partial}{\partial x}\big(\mu(x)\phi(x)\big),\qquad\phi\in C^2(\mathbb{R}).
\end{align}
Note that with no loss of generality we can take $\mathcal{R}$ so that \eqref{continuity05} is well defined and continuous on $\overline{\mathcal{R}}$ by (\textbf{A.1}).

Since $G\in C^{1,2}(\overline{\mathcal{R}})$ there exists a continuous function $\theta_G(|\cdot|)$ such that
\begin{align}\label{continuity06}
\big|V(s,y)-G(s,y)\big|&\le\big|V(s,y)-V(t_0,y)\big|+\big|G(t_0,y)-G(s,y)\big|\nonumber\\
&\le\theta_1(|y|)(t_0-s)^{\alpha/2}+\theta_G(|y|)(t_0-s)\qquad\textrm{for all $(s,y)\in\overline{\mathcal{R}}$}
\end{align}
by \eqref{Vcont}. We may consider $|t_0-t|<1$ and hence \eqref{continuity06} holds with the right-hand side replaced by $\big(\theta_1(|y|)+\theta_G(|y|)\big)(t_0-s)^{\alpha/2}$. We set $\vartheta(y):=\theta_1(|y|)+\theta_G(|y|)$ and use $V\ge G$ and \eqref{continuity06} to obtain
\begin{align}\label{continuity07}
G(s,y)\le V(s,y)\le G(s,y)+\vartheta(y)(t_0-s)^{\alpha/2}\quad\textrm{in $\overline{\mathcal{R}}$}.
\end{align}
For any $s\in(t,t_0)$ we deduce from \eqref{continuity07} that
\begin{align}\label{continuity08}
\int^{x_2}_{x_1}&{V(s,y){\mathbb{L}_X}^*\psi(y)dy}\nonumber\\
=&\int^{x_2}_{x_1}{I_{\{{\mathbb{L}_X}^*\psi\ge 0\}}(y)V(s,y){\mathbb{L}_X}^*\psi(y)dy}+\int^{x_2}_{x_1}{I_{\{{\mathbb{L}_X}^*\psi< 0\}}(y)V(s,y){\mathbb{L}_X}^*\psi(y)dy}\nonumber\\
\le&\int^{x_2}_{x_1}{I_{\{{\mathbb{L}_X}^*\psi\ge 0\}}(y)G(s,y){\mathbb{L}_X}^*\psi(y)dy}+\int^{x_2}_{x_1}{I_{\{{\mathbb{L}_X}^*\psi< 0\}}(y)G(s,y){\mathbb{L}_X}^*\psi(y)dy}\nonumber\\
&+(t_0-s)^{\alpha/2}\int^{x_2}_{x_1}{I_{\{{\mathbb{L}_X}^*\psi\ge 0\}}(y)\vartheta(y){\mathbb{L}_X}^*\psi(y)dy}\nonumber\\
=&\int^{x_2}_{x_1}{G(s,y){\mathbb{L}_X}^*\psi(y)dy}+(t_0-s)^{\alpha/2}\int^{x_2}_{x_1}{I_{\{{\mathbb{L}_X}^*\psi\ge 0\}}(y)\vartheta(y){\mathbb{L}_X}^*\psi(y)dy}\nonumber\\
=&\int^{x_2}_{x_1}{\mathbb{L}_XG(s,y)\psi(y)dy}+(t_0-s)^{\alpha/2}\int^{x_2}_{x_1}{I_{\{{\mathbb{L}_X}^*\psi\ge 0\}}(y)\vartheta(y){\mathbb{L}_X}^*\psi(y)dy}
\end{align}
by integration by parts. Note that the last term is strictly positive by arbitrariness of $\psi$. We define $\gamma\equiv\gamma(\psi;x_1,x_2)>0$ by
\begin{align}\label{continuity09}
\gamma:=\int^{x_2}_{x_1}{I_{\{{\mathbb{L}_X}^*\psi\ge 0\}}(y)\vartheta(y){\mathbb{L}_X}^*\psi(y)dy}.
\end{align}

Now, from \eqref{continuity02}, \eqref{continuity03}, \eqref{continuity04} and \eqref{continuity08} we obtain
\begin{align}\label{continuity10}
0\le \int^{t_0}_{t}\hs{-5pt}\int^{x_2}_{x_1}{\Big(G_t(s,y)+\mathbb{L}_XG(s,y)\Big)\psi(y)dy\,ds}+\gamma(t_0-t)^{1+\alpha/2}.
\end{align}
One may observe from (\textbf{C.1}) that the first integral in \eqref{continuity10} must be strictly ne\-ga\-tive. 
In fact, recalling that $\int^{x_2}_{x_1}\psi(y)dy=1$, there exists $\ell>0$ depending on $x_1$, $x_2$, such that
\begin{align}\label{continuity12}
0\le -\ell(t_0-t)+\gamma(t_0-t)^{1+\alpha/2}
\end{align}
by \eqref{regG-bdry} and \eqref{continuity10}. In the limit as $t\uparrow t_0$ we inevitably reach a contradiction as the positive term vanishes more rapidly than the negative one and hence the jump may not occur.
\end{proof}

The proof above did not require any probabilistic arguments and it follows from simple PDE results and the regularity assumptions on $V$ and $G$. It is sometimes useful to relax condition (\textbf{C.3}) and replace it by (\textbf{C.4}). The latter is in fact easier to verify than (\textbf{C.3}) and holds for a wide class of gain functions $G$. Although the proof of continuity becomes slightly more involved, bounds similar to \eqref{continuity07} may be retrieved by means of purely probabilistic arguments.

\begin{theorem}\label{cor-cont}
Let (\textbf{A.1})--(\textbf{A.4}) and Assumption \ref{assfreeb} hold. Let $[t_1,t_2]\subset [0,T]$ be a time interval where the free-boundary is \emph{increasing}. Then, under conditions (\textbf{C.1}) and (\textbf{C.4}) the free-boundary $t\mapsto b(t)$ is continuous on $[t_1,t_2]$.
\end{theorem}
\begin{proof}
We recall once more that since $\DD$ is closed and $b$ is increasing, standard arguments (see, e.g.~\cite{Jacka}) imply that $b$ is \emph{right-continuous}. To prove continuity we argue again by contradiction and assume that there exists $t_0\in(t_1,t_2]$ where a discontinuity of $b$ occurs and $b(t_0-)<b(t_0)$. We define an open bounded domain $\mathcal{U}\subset \CC$, $\mathcal{U}:=(\bar{t},t_0)\times(x^0_1,x^0_2)$ with $x^0_1$ and $x^0_2$ such that $b(t_0-)<x^0_1<x^0_2<b(t_0)$ and arbitrary $\bar{t}\in(t_1,t_0)$. Its parabolic boundary $\partial_P\mathcal{U}$ is formed by the horizontal lines $[\bar{t},t_0)\times\{x^0_i\}$, $i=1,2$ and by the vertical line $\{t_0\}\times[x^0_1,x^0_2]$.  Hence, the value function $V$ is (unique) classical solution of the boundary value problem
\begin{eqnarray}\label{recont01}
\begin{array}{ll}
\vspace{+5pt}
u_t+\mathbb{L}_Xu=0 & \textrm{in $\mathcal{U}$}\\
\vspace{+5pt}
u=V & \textrm{on $\partial_P\mathcal{U}$}
\end{array}
\end{eqnarray}
by \eqref{freeb-pr}.

Fix $\eta_0>0$ such that $2\eta_0<\min\{|x^0_2-x^0_1|,2\}$, take $x_1$ and $x_2$ such that $(x_1,x_2)\subset (x^0_1+\eta_0,x^0_2-\eta_0)$ and take an arbitrary $t\in(\bar{t},t_0)$. Now define another open bounded domain $\mathcal{R}\subset\mathcal{U}$ by $\mathcal{R}:=(t,t_0)\times(x_1,x_2)$ with parabolic boundary $\partial_P\mathcal{R}$ formed by the horizontal lines $(t,t_0)\times\{x_i\}$, $i=1,2$ and by the vertical line $\{t_0\}\times[x_1,x_2]$ and such that \eqref{regG-bdry} holds.

Take $\psi\ge0$ arbitrary in $C^{\infty}_c([x_1,x_2])$ and such that $\int^{x_2}_{x_1}{\psi(y)dy}=1$, multiply the first equation in \eqref{recont01} (with ${V}$ instead of $u$) by $\psi$ and integrate over $(t,t_0)\times(x_1,x_2)$. It follows
\begin{align}\label{recont02}
\int^{t_0}_{t}\hs{-5pt}\int^{x_2}_{x_1}{\Big({V}_t(s,y)+\mathbb{L}_XV(s,y)\Big)\psi(y)dy\,ds}=0.
\end{align}
Of course we would like to reproduce here arguments similar to those adopted in \eqref{continuity02}--\eqref{continuity10} to find a contradiction in \eqref{recont02}. However, in order to do so we need a bound similar to \eqref{continuity07}.

Let $\tau_{\mathcal{U}}$ denote the first exit time of $(s+r,X_r)_{r\ge0}$ from $\mathcal{U}$ for $s\in[t,t_0)$ and $X_0=y\in(x_1,x_2)$, that is,
\begin{align}\label{recont03}
\tau_{\mathcal{U}}(s,y):=\inf\big\{s\ge0\,:\,(s+r,X_r)\notin\mathcal{U}\big\}\quad\textrm{under $\PP_y$, $y\in(x_1,x_2)$}.
\end{align}
Set $\tau_{\mathcal{U}}\equiv\tau_{\mathcal{U}}(s,y)$ for simplicity. Clearly $\tau_{\mathcal{U}}\le t_0-s\le t_0-t$, $\PP_y$-a.s.\ for $y\in(x_1,x_2)$. Also, the stopping time
\begin{align}\label{recont04}
\tau_*(s,y):=\inf\big\{r\ge0\,:\,X_r\le b(s+r)\big\}\wedge (T-s)
\end{align}
is an optimal stopping time for $V(s,y)$ and $\tau_{\mathcal{U}}\le\tau_*$, $\PP_y$-a.s.\ for $y\in(x_1,x_2)$, from monotonicity of $b$.

Since $\mathcal{U}$ is arbitrary, $G\in C^{1,2}$ in $\overline{\mathcal{U}}$ by (\textbf{A.4}) and we may use It\^o's calculus to obtain
\begin{align}\label{recont05}
G(s,y)=\EE_y\Big[G(s+\rho,X_\rho)-\int_0^\rho{\Big(G_t+\mathbb{L}_XG\Big)(s+r,X_r)dr}\Big]
\end{align}
for all stopping times $\rho\le \tau_{\mathcal{U}}$, $\PP_y$-a.s., $y\in(x_1,x_2)$. Set for simplicity $\tau_*\equiv\tau_*(s,y)$ and take $\rho=\tau_{U}\wedge\tau_*$ in \eqref{recont05}. It follows
\begin{align}\label{recont06}
0\le&V(s,y)-G(s,y)\nonumber\\
=&\EE_y\Big[G(s+\tau_*,X_{\tau_*})-G(s+\tau_{U}\wedge\tau_*,X_{\tau_{\mathcal{U}}\wedge\tau_*})
+\int_0^{\tau_{\mathcal{U}}\wedge\tau_*}{\Big(G_t+\mathbb{L}_XG\Big)(s+r,X_r)dr}\Big]\nonumber\\
\le&\EE_y\Big[G(s+\tau_*,X_{\tau_*})-G(s+\tau_{\mathcal{U}}\wedge\tau_*,X_{\tau_{\mathcal{U}}\wedge\tau_*})
\Big],
\end{align}
where in the last inequality we used \eqref{regG-bdry}. The only non-zero contribution in the last expression of \eqref{recont06} comes from the set $\big\{\tau_*>\tau_{\mathcal{U}}\big\}$ and we have
\begin{align}\label{recont07}
0\le&V(s,y)-G(s,y)\nonumber\\
\le&\EE_y\Big[\Big(G(s+\tau_*,X_{\tau_*})-G(s+\tau_{\mathcal{U}},X_{\tau_{\mathcal{U}}})\Big)I_{\{\tau_*>\tau_{\mathcal{U}}\}}\Big]
\nonumber\\
\le&\EE_y\Big[\Big|G(s+\tau_*,X_{\tau_*})-G(s+\tau_{\mathcal{U}},X_{\tau_{\mathcal{U}}})\Big|^\delta\Big]
^{\frac{1}{\delta}}\PP_y\big(\tau_*>\tau_{\mathcal{U}}\big)^{1-\frac{1}{\delta}}
\end{align}
where we have used H\"older's inequality $\EE|XY|\le\big(\EE |X|^p\big)^{1/p}\big(\EE |Y|^q\big)^{1/q}$ with $p=\delta$ and $q=\frac{\delta}{\delta-1}$. Note that condition (\textbf{C.4}) guarantees that the last term in \eqref{recont07} is well defined.

We observe that if $\tau_*>\tau_{\mathcal{U}}$ then the process exits $\mathcal{U}$ from the upper/lower horizontal boundary strictly before hitting the free-boundary $b$. This also means that $\tau_{\mathcal{U}}<t_0-s$ as otherwise $\tau_{\mathcal{U}}=\tau_*$. Now, recalling that $y\in(x_1,x_2)\subset(x^0_1+\eta_0,x^0_2-\eta_0,)$ we find
\begin{align}\label{recont08}
\Big\{\tau_*>\tau_{\mathcal{U}}\Big\}&\subset\Big\{X_{\tau_{\mathcal{U}}}\le x_1-\eta_0\:\:\:\textrm{or}\:\:\:X_{\tau_{\mathcal{U}}}\ge x_2+\eta_0\Big\}\subset\Big\{\sup_{0\le r\le t_0-s}\big|X_r-y\big|>\eta_0\Big\}.
\end{align}
From \eqref{recont08}, Markov inequality and standard estimates for strong solutions of SDEs (cf.\ for instance \cite{Krylov} Chapter 2, Section 5, Corollary 12 or \cite{Fri06}, Chapter 5, Theorem 2.3 for the case of locally Lipschitz coefficients $\mu$ and $\sigma$) it follows
\begin{align}\label{recont09}
\PP_y\big(\tau_*>\tau_{\mathcal{U}}\big)\le&\PP_y\Big(\sup_{0\le r\le t_0-s}\big|X_r-y\big|>\eta_0\Big)\le\frac{1}{\eta^\beta_0}\EE_y\Big[\sup_{0\le r\le t_0-t}\big|X_r-y\big|^\beta\Big]\nonumber\\
\le&\frac{1}{\eta^\beta_0}C_{T,\beta}\big(1+|y|^\beta\big)\big(t_0-t\big)^{\beta/2}
\end{align}
for a suitable constant $C_{T,\beta}>0$ only depending on $T$ and arbitrary $\beta>0$.

Set $\zeta:=1-1/\delta$. Using \eqref{recont07}, \eqref{recont09} and \eqref{integr-G} we finally obtain
\begin{align}\label{recont10}
0\le V(s,y)-G(s,y)\le 2\kappa(s,y)^{1/\delta}\Big(\frac{1}{\eta^\beta_0}C_{T,\beta}\big(1+|y|^\beta\big)\Big)^{\zeta}\big(t_0-t\big)
^{\zeta\beta/2}
\end{align}
for $y\in(x_1,x_2)$. The inequality \eqref{recont10} provides the analogue of \eqref{continuity07} in the present setting. We repeat the same arguments as in \eqref{continuity01}--\eqref{continuity04} to obtain
\begin{align}\label{recont11}
0\le \int^{t_0}_{t}\hs{-5pt}\int_{x_1}^{x_2}{G_t(s,y)\psi(y)dy\,ds}+
\int^{t_0}_{t}\hs{-5pt}\int_{x_1}^{x_2}{V(s,y){\mathbb{L}_X}^*\psi(y)dy\,ds}
\end{align}
from \eqref{recont02} and use \eqref{recont10}, \eqref{recont11}, condition (\textbf{C.4}) and calculations as in \eqref{continuity08} to find
\begin{align}\label{recont12}
0\le \int^{t_0}_{t}\hs{-5pt}\int_{x_1}^{x_2}{G_t(s,y)\psi(y)dy\,ds}+
\int^{t_0}_{t}\hs{-5pt}\int_{x_1}^{x_2}{{\mathbb{L}_X}G(s,y)\psi(y)dy\,ds}+\frac{c}{\eta^\beta_0}
(t_0-t)^{1+\zeta\beta/2}
\end{align}
with
\begin{align}\label{recont13}
c\equiv c(T,\delta,\eta_0,\beta,x^0_1,x^0_2,\psi):=2c^\infty_\psi\Big[C_{T,\beta}\big(1+\max\{|x^0_1|,|x^0_2|\}^\beta\big)\Big]
^{\zeta}\sup_{0\le s\le T}\int^{x^0_2}_{x^0_1}\kappa^{1/\delta}(s,y)dy
\end{align}
and $c^\infty_\psi\equiv c^\infty_\psi(x^0_1,x^0_2):=\sup_{y\in[x^0_1,x^0_2]}\big|{\mathbb{L}}^*_X\psi(y)\big|$. Note that $c^\infty_\psi<+\infty$ as $\mu$, $\sigma$ and $\psi$ are $C^2$ in $[x^0_1,x^0_2]$.
As usual $G_t+\mathbb{L}_XG<-\ell$ in $\mathcal{U}$ for some constant $\ell>0$ by (\textbf{C.1}) and hence
\begin{align}\label{recont14}
0\le -\ell(t_0-t)+\frac{c}{\eta^\beta_0}
(t_0-t)^{1+\zeta\beta/2}
\end{align}
by \eqref{recont12}. For fixed $\delta$, $\eta_0$, $\beta$, $\mathcal{U}$ and $\psi$, \eqref{recont14} leads to a contradiction in the limit as $t\uparrow t_0$.
\end{proof}

When Assumption \ref{assfreeb} holds and the free-boundary is decreasing, arguments above seem to break down. This is mainly due to the fact that if a jump occurs at $t_0$ and the diffusion $(t,X)$ starts either from a point $(t_0+\eps,x)$ with $\eps>0$ or from $(t_0-\eps,x)$ with $x>b(t_0-)$ then, as time elapses it will move away from the discontinuity. On the contrary, in theorems above and in particular in Theorem \ref{cor-cont} we crucially relied on the fact that $(t,X)$ moves towards the jump. Intuitively it might happen that the diffusion does not ``see'' discontinuities of a decreasing boundary and hence the rationale above may not be adopted. Our first natural attempt to overcome this difficulty was to study the time reversed process $(t-s,X_s)_{s\ge0}$ but we realised that this introduces seemingly harder complications to deal with. An approach based on PDE results and condition (\textbf{C.2}) allows instead to find continuity of $b$ again.

\begin{theorem}\label{th-decr}
Let (\textbf{A.1})--(\textbf{A.4}) and Assumption \ref{assfreeb} hold. Let $[t_1,t_2]\subset [0,T]$ be a time interval where the free-boundary is \emph{decreasing}. Then, under condition (\textbf{C.2}) the free-boundary $t\mapsto b(t)$ is continuous on $[t_1,t_2]$.
\end{theorem}
\begin{proof}
In this case adapting standard arguments (see, e.g.~\cite{Jacka}) one finds that $b$ is \emph{left-continuous} as $\DD$ is closed and $b$ is decreasing. Assume that there exists $t_0\in[t_1,t_2)$ such that a discontinuity of $b$ occurs. That is, $b(t_0+)<b(t_0)$, where $b(t_0+)$ denotes the right limit of the boundary at $t_0$. Take $x_1$ and $x_2$ such that $b(t_0+)<x_1<x_2<b(t_0)$ and $t^\prime\in(t_0,t_2)$; then, define once more an open bounded domain $\mathcal{R}\subset \CC$ with $\mathcal{R}:=(t_0,t^\prime)\times(x_1,x_2)$. Its parabolic boundary $\partial_P\mathcal{R}$ is formed by the horizontal lines $(t_0,t^\prime)\times\{x_i\}$, $i=1,2$, and by the vertical line $\{t^\prime\}\times[x_1,x_2]$ (note that in this setting $\CC$ lies on the right of the vertical segment $[b(t_0+),b(t_0)]$).

Set $u:=V-G$ and recall the definition of $H$ from \eqref{regG-bdry}. Hence, $u$ is (unique) classical solution of the boundary value problem
\begin{eqnarray}\label{rerecont01}
\begin{array}{ll}
\vspace{+5pt}
u_t+\mathbb{L}_Xu=-H & \textrm{in $\mathcal{R}$}\\
\vspace{+5pt}
u=V-G & \textrm{on $\partial_P\mathcal{R}$}
\end{array}
\end{eqnarray}
by \eqref{freeb-pr} and (\textbf{A.4}). Define a differential operator $\mathcal{A}$ by
\begin{align}\label{rerecont02}
\mathcal{A}f(x):=\frac{1}{2}\sigma^2(x)f^{\prime\prime}(x)+\big(\sigma\,\sigma^\prime\,(x)+\mu(x)\big)f^\prime(x)+
\mu^\prime(x)f(x),
\qquad x\in\mathcal{R},\,f\in C^2(\mathcal{R}).
\end{align}
Now, conditions (\textbf{A.1}), (\textbf{A.4}) and (\textbf{C.2}) imply that $u_{tx}$ and $u_{xxx}$ exist and are continuous in $\mathcal{R}$ (cf.~\cite{Fri08}, Chapter 3, Theorem 10). We differentiate the first equation in \eqref{recont01} with respect to $x$ and set $\bar{u}:=u_x$ to obtain
\begin{align}\label{rerecont03}
\bar{u}_t+\mathcal{A}\bar{u}=-H_x \quad \textrm{in $\mathcal{R}$},
\end{align}
with $\mathcal{A}$ as in \eqref{rerecont02}. Take $\psi\ge0$ arbitrary in $C^{\infty}_c([x_1,x_2])$ and such that $\int^{x_2}_{x_1}{\psi(y)dy}=1$. We define a function $F_\psi:(t_0,t^\prime)\to\mathbb{R}$ by
\begin{align}\label{rerecont04}
F_\psi(t):=\int^{x_2}_{x_1}{\bar{u}_t(t,x)\psi(x)dx}
\end{align}
and \eqref{rerecont03} gives
\begin{align}\label{rerecont05}
F_\psi(t)=&-\int^{x_2}_{x_1}{\Big(H_x(t,x)+\mathcal{A}\bar{u}(t,x)\Big)\psi(x)dx}\nonumber\\
=&-\int^{x_2}_{x_1}{H_x(t,x)\psi(x)dx}-\int^{x_2}_{x_1}{\bar{u}(t,x)\mathcal{A}^*\psi(x)dx}\\
=&-\int^{x_2}_{x_1}{H_x(t,x)\psi(x)dx}+\int^{x_2}_{x_1}{u(t,x)\frac{\partial}{\partial x}\big(\mathcal{A}^*\psi\big)(x)dx},\nonumber
\end{align}
where $\mathcal{A}^*$ is the formal adjoint of $\mathcal{A}$ and $\frac{\partial}{\partial x}\big(\mathcal{A}^*\psi\big)\in C_c([x_1,x_2])$ by $(\textbf{A.1})$ and arbitrariness of $[x_1,x_2]$. It follows from \eqref{rerecont05} that $F_\psi$ is continuous on $(t_0,t^\prime)$ and the right-limit of $F_\psi$ at $t_0$ exists and it is
\begin{align}\label{rerecont06}
F_\psi(t_0+):=\lim_{t\downarrow t_0}{F_\psi(t)}=-\int^{x_2}_{x_1}{H_x(t_0,x)\psi(x)dx}
\end{align}
by using that $u\in C(\overline{\mathcal{R}})$ and $u(t_0,x)\equiv 0$ for $x\in[x_1,x_2]$. Therefore, from \eqref{regH-bdry} we obtain $F_{\psi}(t_0+)\le -\ell$ for suitable $\ell>0$ and continuity of $F_\psi$ implies that there exists $\eps>0$ such that $F_\psi(t)<-\ell/2$ for all $t\in(t_0,t_0+\eps)$. Set $0<\delta<\eps$, then \eqref{rerecont04}, Fubini's Theorem and an integration by parts give
\begin{align}\label{rerecont07}
-\frac{\ell}{2}(\eps-\delta)>&\int^\eps_\delta{F_\psi(t_0+s)ds}=\int^{x_2}_{x_1}{\big[\bar{u}(t_0+\eps,x)-
\bar{u}(t_0+\delta,x)\big]\psi(x)dx}\nonumber\\
=&\int^{x_2}_{x_1}{\bar{u}(t_0+\eps,x)\psi(x)dx}+\int^{x_2}_{x_1}{u(t_0+\delta,x)\psi^\prime(x)dx}.
\end{align}
Now, taking limits as $\delta\to0$, using dominated convergence and recalling that $\bar{u}=u_x$, $u\in C(\overline{\mathcal{R}})$, $u(t_0,x)\equiv 0$ for $x\in[x_1,x_2]$ we obtain
\begin{align}\label{rerecont08}
-\frac{\ell}{2}\eps\ge&\int^{x_2}_{x_1}{\bar{u}(t_0+\eps,x)\psi(x)dx}=\int^{x_2}_{x_1}{\big[V_x-G_x\big](t_0+\eps,x)
\psi(x)dx}.
\end{align}
Since $\psi\ge0$, then \eqref{VxGx} and \eqref{rerecont08} lead to a contradiction; hence $b$ must be continuous.
\end{proof}

It easy to see that from Theorems \ref{th-cont}--\ref{th-decr} it follows
\begin{coroll}
Let (\textbf{A.1})--(\textbf{A.4}), Assumption \ref{assfreeb} and conditions (\textbf{C.1}), (\textbf{C.2}) hold. Assume that the free-boundary is piecewise monotone on $[0,T]$ and that either (\textbf{C.3}) or (\textbf{C.4}) holds. Then $b$ is continuous on $[0,T]$.
\end{coroll}

The special case of a time-independent gain function may be treated separately. In fact, in that case conditions (\textbf{C.2}), (\textbf{C.3}) and (\textbf{C.4}) may be dropped and continuity is obtained in a very general setting. We prove this claim in the next Proposition, for completeness.
\begin{prop}\label{time-indep}
Assume (\textbf{A.1}) and that $G:\mathbb{R}\to\mathbb{R}$ is time independent and it meets (\textbf{A.2}). Assume also that the value function
\begin{align}\label{time-ind-V}
V(t,x):=\sup_{0\le\tau\le T-t}\EE_x\Big[G(X_\tau)\Big]
\end{align}
fulfils (\textbf{A.3}) and that there exists a free-boundary $\big\{b(t)$, $0\le t\le T\big\}$ such that Assumption \ref{assfreeb}, (\textbf{A.4}) and (\textbf{C.1}) hold. Then $t\mapsto b(t)$ is increasing and continuous on $[0,T]$.
\end{prop}
\begin{proof}
To show that $t\mapsto b(t)$ is monotone increasing we use standard arguments (cf.~\cite{Jacka}, for instance). Since $G$ does not depend on time, the mapping $t\mapsto V(t,x)$ is decreasing for any $x\in\mathbb{R}$. Take $(t_0,x_0)\in\DD$ and $t>t_0$, then $V(t,x_0)\le V(t_0,x_0)=G(x_0)$ and hence $(t,x_0)\in\DD$ for all $t>t_0$. Closedness of $\DD$ implies that $b(t)$ is also right-continuous.

To prove continuity, assume that there exists $t_0\in(0,T]$ such that $b(t_0-)<b(t_0)$; construct a rectangular domain $\mathcal{R}$ with parabolic boundary $\partial_P\mathcal{R}$ as in the proof of Theorem \ref{th-cont}, then $V$ is the (unique) classical solution of the Cauchy-Dirichlet problem
\begin{eqnarray}\label{cont-TI01}
\begin{array}{ll}
\vspace{+5pt}
u_t+\mathbb{L}_Xu=0 & \textrm{in $\mathcal{R}$}\\
\vspace{+5pt}
u=V & \textrm{on $\partial_P\mathcal{R}$}.
\end{array}
\end{eqnarray}

Take $\psi\ge0$ arbitrary in $C^{\infty}_c([x_1,x_2])$ with $\int^{x_2}_{x_1}{\psi(y)dy}=1$, multiply the first equation in \eqref{cont-TI01} (with $V$ instead of $u$) by $\psi$ and integrate over $[x_1,x_2]$. Then
\begin{align}\label{cont-TI02}
\int_{x_1}^{x_2}{{V}_t(s,y)\psi(y)dy}=-\int_{x_1}^{x_2}{\mathbb{L}_X{V}(s,y)\psi(y)dy}\qquad\textrm{for all $s\in(t^\prime,t_0)$}
\end{align}
and, integrating by parts the term on the right-hand side of \eqref{cont-TI02}, we obtain
\begin{align}\label{cont-TI03}
\int_{x_1}^{x_2}{{V}_t(s,y)\psi(y)dy}=-\int_{x_1}^{x_2}{{V}(s,y){\mathbb{L}_X}^*\psi(y)dy}
\qquad\textrm{for all $s\in(t^\prime,t_0)$}.
\end{align}

The left-hand side of \eqref{cont-TI03} is negative since $V$ is decreasing in time and $\psi\ge0$. Then, we take the limit as $s\uparrow t_0$ in the right-hand side of \eqref{cont-TI03} and use dominated convergence, continuity of $V$ and boundary condition $V(t_0,y)=G(y)$ to obtain
\begin{align}\label{cont-TI04}
0\ge-\int_{x_1}^{x_2}{{V}(t_0,y){\mathbb{L}_X}^*\psi(y)dy}=
-\int_{x_1}^{x_2}{G(y){\mathbb{L}_X}^*\psi(y)dy}=-\int_{x_1}^{x_2}{\mathbb{L}_XG(y)\psi(y)dy}.
\end{align}
\noindent There exists $\ell>0$ depending only on $\mathcal{R}$ such that $\mathbb{L}_XG(y)<-\ell$ in $[x_1,x_2]$, by \eqref{regG-bdry} and then we find the contradiction
\begin{align}\label{cont-TI05}
0\ge\ell\int_{x_1}^{x_2}{\psi(y)dy}=\ell>0
\end{align}
by \eqref{cont-TI04}.
\end{proof}

A basic example of a more general optimal stopping problem may be considered by taking
\begin{align}\label{gen01}
G(t,x):=g(t,x)I_{\{t<T\}}+h(x)I_{\{t=T\}}
\end{align}
in \eqref{Vfun} with $g$ and $h$ bounded and continuous. A probabilistic proof of the existence of an optimal stopping time in this setting may be found in \cite{PaSt10} and the continuation set is $\CC:=\big\{V>g\big\}$. Then, replacing assumptions on $G$ with analogous ones for $g$ one may use the same arguments as above to show that $b(t)$ as in Assumption \ref{assfreeb} is continuous on the open interval $(0,T)$. However, continuity at the maturity $T$ may break down and it should be studied again on a case by case basis. Note that if $b(t)$ is decreasing on $(t_1,T]$ for some $t_1<T$, then it is also continuous at $T$ since it is left-continuous on the interval. On the other hand, for an increasing boundary on $(t_1,T]$ one may easily check, proceeding as in the proofs of Theorems \ref{th-cont} and \ref{cor-cont}, that sufficient conditions for continuity at $T$ are: \emph{i)} $\mathbb{L}_Xh$ well defined and piecewise continuous, \emph{ii)} $\mathbb{L}_Xh<-\ell_\eps$ for $x\le b(T)$.

All results of this paper naturally extend to the case of discounted gain functions and in presence of running costs. In fact, if we take for instance a positive, continuous discount function $r(x)$ and a positive, continuous cost function $C(t,x)$ such that
\begin{align}\label{cost}
\EE_x\Big[\int^{T}_0C(s,X_s)ds\Big]<+\infty\quad\textrm{for $x\in\mathbb{R}$}
\end{align}
we may define the optimal stopping problem
\begin{align}\label{cost02}
V(t,x):=\sup_{0\le\tau\le T-t}\EE_x\Big[e^{-\int^\tau_0{r(X_s)ds}}G(t+\tau,X_\tau)-\int^\tau_0{e^{-\int_0^s{r(X_u)du}}C(t+s,X_s)ds}\Big].
\end{align}
If $(\textbf{A.2})$ and $(\textbf{A.3})$ hold then $\tau_*$ as in \eqref{taustar} is optimal. If we now modify conditions (\textbf{C.1}) and (\textbf{C.2}) by using $G_t+\mathbb{L}_XG-rG-C$ instead of $G_t+\mathbb{L}_XG$ and assume for instance that
\begin{align}\label{cost03}
\xi(t,x):=\EE_x\Big[\int^{T-t}_0\big|C(t+s,X_s)\big|^\delta ds\Big]
\end{align}
is locally bounded on $[0,T]\times\mathbb{R}$, for $\delta>1$ as in (\textbf{C.4}), all conclusions above remain true (this is the case for instance of \cite{Ch-Haus09}).

\vs{+10pt}
\noindent\textbf{Acknowledgments:} I wish to thank G.~Ferrari and G.~Peskir for useful discussions and comments.


\begin{thebibliography}{99}
\bibitem{BF77} \textsc{Bensoussan, A.\ \emph{and} Friedman, A.}\ (1977). Non-Zero Sum Stochastic Differential Games with Stopping Times and Free-Bounadry Problems. \emph{Trans.~Amer.~Math.~Soc.}~\textbf{231} No.\ 2 (275--327).
\bibitem{BL82} \textsc{Bensoussan, A.\ \emph{and} Lions, J.L.}\ (1982). Applications of Variational Inequalities in Stochastic Control. North Holland.
\bibitem{Bl06} \textsc{Blanchet, A.}\ (2006). On the Regularity of the Free-Boundary in Parabolic Obstacle Problems. Application to American Options. \emph{Nonlinear Anal.}~\textbf{65} (1362--1378).
\bibitem{CDH67} \textsc{Cannon, J.~R.~\emph{and} Denson Hill, C.}~(1967). On the Infinite Differentiability of the Free Boundary in a Stefan problem. \emph{J.~Math.~Anal.~Appl.}~\textbf{22} (385--397).
\bibitem{CJM92} \textsc{Carr, P., Jarrow, R. \emph{and} Myneni, R.}~(1992). Alternative Characterizations of American Put Options. \emph{Math.~Finance} \textbf{2} (78--106).
\bibitem{CC07} \textsc{Chen, X.\ \emph{and} Chadam, J.}\ (2007). A Mathematical Analysis of The Optimal Exercise Boundary for American Put Options. \emph{SIAM J.~Math.~Anal.}~\textbf{38} No.\ 5 (1613--1641).
\bibitem{CCJZ08} \textsc{Chen, X., Chadam, J., Jiang, L.\ \emph{and} Zheng, W.}\ (2008). Convexity of the Exercise Boundary of the American Put Options on a Zero Dividend Asset. \emph{Math.~Finance} \textbf{18} No.\ 1 (185--197).
\bibitem{CDeAV13} \textsc{Chiarolla, M.\ B., De Angelis, T.~\emph{and} Valdivia, I.}~(2013). On Free-Boundary Problems Arising in Partecipating Policies with Surrender Options. \emph{Preprint}.
\bibitem{Ch-Haus09} \textsc{Chiarolla, M.\ B.~\emph{and} Haussmann, U.\ G.}~(2009). On a Stochastic, Irreversible Investment Problem. \emph{SIAM J.\ Control Optim.}~\textbf{48} No.\ 2 (438--462).
\bibitem{Cox-Pe13} \textsc{Cox, A.~M.~\emph{and} Peskir, G.}~(2013). Embedding Laws in Diffusions by Functions of Time. \emph{Preprint}. \href{http://arxiv.org/abs/1201.5321v2}{http://arxiv.org/abs/1201.5321v2}.
\bibitem{DeA-Fe13} \textsc{De Angelis, T.~\emph{and} Ferrari, G.}~(2013). A Stochastic Reversible Investment Problem on a Finite Time-Horizon: Free Boundary Analysis. \emph{Preprint}. \\ \href{http://arxiv.org/abs/1303.6189v3}{http://arxiv.org/abs/1303.6189v3}
\bibitem{DuT-Pe07} \textsc{Du Toit, J.~\emph{and} Peskir, G.}~(2007). The Trap of Complacency in Predicting the Maximum. \emph{Ann.~Appl.~Prob.}~\textbf{35} No.\ 1 (340--365).
\bibitem{Eks04} \textsc{Ekstrom, E.}\ (2004). Convexity of the Optimal Stopping Boundary for the American Put Option. \emph{J. Math.~Anal.~Appl.}~\textbf{299} (147--156).
\bibitem{Fri73} \textsc{Friedman, A.}\ (1973). Regularity Theorems for Variational Inequalities in Unbounded Domains and Applications to Stopping Time Problems. \emph{Arch.~Ration.~Mech.~Anal.}~\textbf{52} No.\ 2 (134--160).
\bibitem{Fri75} \textsc{Friedman, A.}\ (1975). Parabolic Variational Inequalities in one Space Dimension and Smoothness of the Free-Boundary. \emph{J.~Funct.~Anal.}~\textbf{18} (151--176).
\bibitem{Fri06} \textsc{Friedman, A.}\ (2006). Stochastic Differential Equations and Applications. \emph{Dover Publications, New York}.
\bibitem{Fri08} \textsc{Friedman, A.}\ (2008). Partial Differential Equations of Parabolic Type. \emph{Dover Publications, New York}.
\bibitem{Jacka} \textsc{Jacka, S.D.}~(1991). Optimal Stopping and the American Put. \emph{Math.~Finance} \textbf{1} (1-14).
\bibitem{JL92} \textsc{Jacka, S.D.~\emph{and} Lynn, J.~R.}~(1992). Finite-Horizon Optimal Stopping, Obstacle Problems and the Shape of the Continuation Region. \emph{Stoch.~Stoch.~Rep.} \textbf{39} (25-42).
\bibitem{KIM} \textsc{Kim, I.~J.}~(1990). The Analytic Valuation of American Options. \emph{Rev.~Financial Stud.}~\textbf{3} (547--572).
\bibitem{Kot73} \textsc{Kotlow, D.B.}\ (1973). A Free-Boundary Problem Connected with the Optimal Stopping Problem for Diffusion Processes. \emph{Trans.~Amer.~Math.~Soc.}~\textbf{184} (457--478).
\bibitem{Krylov} \textsc{Krylov, N.V.}~(2009). Controlled Diffusion Processes. \emph{Springer-Verlag, Berlin Heidelberg}.
\bibitem{McCon91} \textsc{Mc Connell, T.~R.}~(1991). The two-sided Stefan problem with a spatially dependent latent heat. \emph{Trans.~Amer.~Math.~Soc.}~\textbf{326} (669--699).
\bibitem{McK65} \textsc{McKean, H.~P., Jr.~}(1965). Appendix: A Free Boundary problem for the Heat Equation arising from a problem of Mathematical Economics. \emph{Ind.~Management Rev.~}\textbf{6} (32--39).
\bibitem{Myn92} \textsc{Myneni, R.}~(1992). The Pricing of the American Option. \emph{Ann.~Appl.~Probab.}~\textbf{2} (1--23).
\bibitem{PaSt10} \textsc{Palczewski, J. \emph{and} Stettner, L.}\ (2010).
Optimal sopping of time discontinuous functionals with applications to impulse control with delay. \emph{SIAM J. Control Optim.}~\textbf{48} No.\ 8, (4874--4909)
\bibitem{Pe05} \textsc{Peskir, G.~}(2005). On the American Option problem. \emph{Math.~Finance} \textbf{15} (1) (169--181).
\bibitem{Pe05b} \textsc{Peskir, G.~}(2005). A Change-of-variable Formula with Local Time on Curves. \emph{J.~Theoret.~Probab.}~\textbf{18} (3) (499--535).
\bibitem{Pe07} \textsc{Peskir, G.~}(2007). A Change-of-variable Formula with Local Time on Surfaces. \emph{S\'em.\ de Probab. XL, Lecture Notes in Math.}~Vol.\ {1899}, Springer (69--96).
\bibitem{Pes-Shir} \textsc{Peskir, G.~\emph{and} Shiryaev, A.}~(2006). Optimal Stopping and Free-Boundary Problems. \emph{Lectures in Mathematics, ETH Zurich}.
\bibitem{PeUys05} \textsc{Peskir, G.~\emph{and} Uys, N.}~(2005). On Asian Options of American Type. \emph{Exotic Option Pricing and Advanced L\'evy Models (Eindhoven 2004), John Wiley} (217--235).
\bibitem{Sch76} \textsc{Schaeffer, D.~G.~}(1976). A New Proof of the Infinite Differentiability of the Free Boundary in the Stefan Problem. \emph{J.~Differential Equations} \textbf{20} (266--269).
\bibitem{Shir} \textsc{Shiryaev, A.}~(1978). Optimal Stopping Rules. \emph{Springer-Verlag, Berlin Heidelberg}.
\bibitem{VM} \textsc{Van Moerbeke, P.}\ (1976). On optimal stopping and free-boundary problems. \emph{Arch.~Ration.~Mech.~Anal.}~\textbf{60} (101--148).
\bibitem{YYZ12} \textsc{Yam, S.C.P., Yung, S.P.~\emph{and} Zhou, W.}~(2012). Game Call Options Revisited. \emph{Math.~Finance}, doi: 10.1111/mafi.12000 (1--34).

\end{thebibliography}
\end{document}